\documentclass[10pt]{article}
\usepackage{amssymb}
\usepackage{amsthm}
\usepackage{latexsym}
\usepackage{comment}
\usepackage{hyperref}
\usepackage{xypic}
\usepackage{graphicx}

\title{Dyadic Torsion of 2-Dimensional Hyperelliptic Jacobians}
\author{Jeffrey Yelton}

\newtheorem{thm}{Theorem}[section]
\newtheorem{prop}[thm]{Proposition}

\newtheorem{lemma}[thm]{Lemma}

\newtheorem{dfn}[thm]{Definition}

\theoremstyle{definition} \newtheorem{rmk}[thm]{Remark}

\newcommand{\qq}{\mathbb{Q}}
\newcommand{\zz}{\mathbb{Z}}
\newcommand{\ff}{\mathbb{F}}

\newcommand{\pp}{\mathbb{P}}
\newcommand{\Gal}{\mathrm{Gal}}
\newcommand{\Div}{\mathrm{Div}}
\newcommand{\Aut}{\mathrm{Aut}}
\newcommand{\GL}{\mathrm{GL}}
\newcommand{\Sp}{\mathrm{Sp}}
\newcommand{\GSp}{\mathrm{GSp}}
\newcommand{\Pic}{\mathrm{Pic}}
\newcommand{\Ri}{\mathrm{Ri}}

\begin{document}

\maketitle

\begin{abstract}

Let $k$ be a field of characteristic $0$, and let $\alpha_{1}$, $\alpha_{2}$, ... , $\alpha_{5}$ be algebraically independent and transcendental over $k$.  Let $K$ be the transcendental extension of $k$ obtained by adjoining the elementary symmetric functions of the $\alpha_{i}$'s.  Let $J$ be the Jacobian of the hyperelliptic curve defined over $K$ which is given by the equation $y^{2} = \prod_{i = 1}^{5} (x - \alpha_{i})$.  We define a tower of field extensions $K = K_{0}' \subset K_{1}' \subset K_{2}' \subset ...$ by giving recursive formulas for the generators of each $K_{n}'$ over $K_{n - 1}'$, and let $K_{\infty}' = \bigcup_{n = 0}^{\infty} K_{n}'$.  We show that $K_{\infty}'(\mu_{2})$ is the subextension of the field $K(J[2^{\infty}]) := \bigcup_{n = 0}^{\infty} K(E[2^{n}])$ corresponding to a central order-$2$ Galois subgroup of $\Gal(K(J[2^{\infty}]) / K(\mu_{2}))$, and a generator of $K(J[2^{\infty}])$ over $K_{\infty}'(\mu_{2})$ is given.

\end{abstract}

\section{Main definitions and results}

Let $k$ be a field of characteristic $0$.  Let $K$ be the transcendental extension of $k$ obtained by adjoining the coefficients of the quintic polynomial $\prod_{i = 1}^{5} (x - \alpha_{i})$, where the $\alpha_{i}$'s are independent and transcendental over $k$.  Fix an algebraic closure $\bar{K}$ of $K$.  Suppose that $C$ is a smooth projective model of an affine hyperelliptic curve over $K$ given by the Weierstrass equation
\begin{equation}\label{hyperelliptic model genus 2}
y^{2} = \prod_{i = 1}^{5} (x - \alpha_{i}).
\end{equation}
Write $J$ for the Jacobian of $C$, which is also defined over $K$.  For any integer $n \geq 0$; let $J[2^{n}]$ be the subgroup of $J(\bar{K})$ of $2^{n}$-torsion points; and let $K_{n}$ be the extension of $K$ over which they are defined.  (Note that $K_{0} = K$.)  Further, denote by $J[2^{\infty}]$ the subgroup of all $2$-power torsion points and denote by $K_{\infty}$ the minimal (algebraic) extension of $K$ over which they are defined.

Let $T$ be a $15$-regular rooted tree; let $|T|$ denote the set of vertices of $T$; and let $v_{0} \in |T|$ denote the root vertex.  Since $T$ is a tree, one may define the ``distance" between two vertices in $|T|$ to be the number of edges in a simple path connecting them.  For any integer $n \geq 0$, let $|T|_{n}$ (respectively $|T|_{\leq n}$, $T_{\geq n}$) denote the subset of vertices of $T$ which are of distance $n$ (respectively $\leq n$, $\geq n$) from the root $v_{0}$.  The fact that $T$ is a tree also implies that each vertex $v \in |T|_{n}$ for $n \geq 1$ has exactly one ``parent", that is, a unique vertex $\tilde{v} \in |T|_{n - 1}$ of distance $1$ from $v$.

Consider the set $\mathcal{R}$ consisting of all ordered triples $(R_{1}, R_{2}, R_{3})$, where the $R_{i}$'s are pairwise disjoint (unordered) $2$-element subsets of $\pp_{\bar{K}}^{1}$.  We say that two such triples $(R_{1}, R_{2}, R_{3}), (S_{1}, S_{2}, S_{3}) \in \mathcal{R}$ are \textit{permutation equivalent} if there is a permutation $\sigma$ on $\{1, 2, 3\}$ such that $S_{i} = R_{\sigma(i)}$ for $i = 1, 2, 3$.  Let $\bar{\mathcal{R}}$ be the set of permutation equivalence classes of such triples; we will write $[(R_{1}, R_{2}, R_{3})] \in \bar{\mathcal{R}}$ for the equivalence class of a triple $(R_{1}, R_{2}, R_{3})$.  If $R = [(R_{1}, R_{2}, R_{3})] \in \bar{\mathcal{R}}$, let $|R| := R_{1} \cup R_{2} \cup R_{3} \subset \pp_{\bar{K}}^{1}$; clearly this set has cardinality $6$ and is well-defined regardless of the choice of representative for $R$.

Define $M : \mathcal{R} \to M_{3}(\bar{K})$ to be the map which sends $(R_{1}, R_{2}, R_{3})$

\noindent$= (\{r_{1, 1}, r_{1, 2}\}, \{r_{2, 1}, r_{2, 2}\}, \{r_{3, 1}, r_{3, 2}\}) \in \mathcal{R}$ to the $3$-by-$3$ matrix $M(R)$ with entries in $\bar{K}$ defined as follows.  If $\infty \notin R_{i}$, then the $i$th row of $M(R)$ is $(r_{i, 1}r_{i, 2}, -(r_{i, 1} + r_{i, 2}), 1)$.  If $\infty \in R_{i}$ and we assume without loss of generality that $r_{i, 2} = \infty$, then the $i$th row of $M(R)$ is $(-r_{i, 1}, 1, 0)$.  

Let $U \subset M_{3}(\bar{K})$ be the subset of matrices $A = (A_{i, j})$ such that $A_{i, 3}x^{2} + A_{i, 2}x + A_{i, 1} \in \bar{K}[x]$ is a squarefree polynomial of degree $1$ or $2$ for each $i$, and such that $A_{i, 3} = 0$ for at most one $i \in \{1, 2, 3\}$.  Note that $M(\mathcal{R}) \subset U$.  Now define $N : U \to \mathcal{R}$ as follows.  Let $A = (A_{i, j})$ be a matrix in $U$.  Then if $A_{i, 3} \neq 0$, let $R_{i}$ be the set of roots of the (squarefree) quadratic polynomial $A_{i, 3}x^{2} + A_{i, 2}x + A_{i, 1} \in \bar{K}[x]$, and if $A_{i, 3} = 0$, let $R_{i} = \{-A_{i, 1} / A_{i, 2}, \infty\}$.  It is easy to check that $N \circ M$ is the identity function on $\mathcal{R}$.

For any matrix $A \in M_{3}(\bar{K})$, let 
$$A^{\vee} = \left( {\begin{array}{ccc}  0 & 0 & 1 \\ 0 & -2 & 0 \\ 1 & 0 & 0 \end{array} } \right) \mathrm{adj}(A),$$
where $\mathrm{adj}(A) := \det(A)A^{-1}$ denotes the adjugate of $A$.

If $(R_{1}, R_{2}, R_{3}) \in \mathcal{R}$ such that $M((R_{1}, R_{2}, R_{3}))^{\vee} \in M(\mathcal{R})$, then we assign $\Ri((R_{1}, R_{2}, R_{3})) = N(M((R_{1}, R_{2}, R_{3}))^{\vee})$.  It turns out (see Proposition \ref{prop: existence of decoration genus 2}(b) below) that $\Ri((R_{1}, R_{2}, R_{3}))$ is defined for any $(R_{1}, R_{2}, R_{3}) \in \mathcal{R}$.  Moreover, this map $\Ri$ clearly respects permutation equivalence, so it descends to a map $\Ri : \bar{\mathcal{R}} \to \bar{\mathcal{R}}$, which we call the \textit{Richelot operator} on $\bar{\mathcal{R}}$.

\begin{dfn}\label{def: decoration genus 2}

A \textit{decoration} on the 15-regular tree $T$ is a map $\Psi : |T|_{\geq 1} \to \bar{\mathcal{R}}$ with the following properties:

a) For any two distinct vertices $w, w' \in |T|_{\geq 1}$ with the same parent vertex, $\Psi(w) \neq \Psi(w')$.

b) For any vertex $w \in |T|_{1}$, $|\Psi(w)| = \{\alpha_{i}\}_{i = 1}^{5} \cup \{\infty\}$.

c) For any vertex $w \in |T|_{\geq 2}$, $|\Psi(w)| = |\Ri(\Psi(\tilde{w}))|$ but $\Psi(w) \neq \Ri(\Psi(\tilde{w}))$.

\end{dfn}

The following proposition shows that there exists a decoration on $T$.

\begin{prop}\label{prop: existence of decoration genus 2}

Let $\Psi$ be a decoration on $T$.  Then we have the following:

a) There are exactly $15$ permutation equivalence classes $R \in \bar{\mathcal{R}}$ with the property that $|R| = \{\alpha_{i}\}_{i = 1}^{5} \cup \{\infty\}$.

b) The Richelot operator $\Ri$ is defined for all permutation classes in $\bar{\mathcal{R}}$ (in particular, $\Ri(\Psi(w))$ is defined for any $w \in |T|_{\geq 1}$).

c) For any $w \in |T|_{\geq 1}$, up to permutation equivalence, there are exactly $14$ permutation equivalence classes $R \in \bar{\mathcal{R}}$ such that $|R| = |\Ri(\Psi(w))|$ but $R \neq \Ri(\Psi(w))$.

\end{prop}

\begin{rmk}

Proposition \ref{prop: existence of decoration genus 2} implies that a decoration exists via the following argument.  For each $N \geq 1$, define $F_{N}$ to be the set of all functions $\Psi : |T|_{\leq N} \setminus \{v_{0}\} \to \bar{K}$ which satisfy Definition \ref{def: decoration genus 2} for $v \in |T|_{\leq N} \setminus \{v_{0}\}$.  Clearly, each $F_{N}$ is finite, and for each $N < N'$, there is a map $F_{N'} \to F_{N}$ by restriction, so it will suffice to show that each $F_{N}$ is nonempty.  Part (a) implies that $F_{1}$ is nonempty, and parts (b) and (c) show that if $F_{N}$ is nonempty, then so is $F_{N + 1}$; thus, $F_{N}$ is nonempty for all $N$ as desired.

\end{rmk}

\begin{proof}

Note that there are exactly $\frac{1}{3!} {6 \choose 2} {4 \choose 2} {2 \choose 2} = 15$ partitions of $6$ objects into $3$ pairs.  It immediately follows that if $S$ is any element of $\bar{\mathcal{R}}$, there are $15$ partitions of the $6$ elements of $|S|$ into $3$ pairs, and thus, there are $15$ permutation equivalence classes $R$ such that $|R| = |S|$.  Part (a) follows from the fact that $\{\alpha_{i}\}_{i = 1}^{5} \cup \{\infty\}$ is a set of $6$ objects.  Let $(R_{1}, R_{2}, R_{3})$ be a representative of a permutation equivalence class in $R \in \bar{\mathcal{R}}$, and let $A = M((R_{1}, R_{2}, R_{3}))$.  Then the polynomial $\prod_{i = 1}^{3} (A_{i, 3}x^{2} + A_{i, 2}x + A_{i, 1}) \in \bar{K}[x]$ is squarefree and has degree $5$ or $6$.  By Lemma 8.4.2 of \cite{smith2005explicit}, $\prod_{i = 1}^{3} ((A^{\vee})_{i, 3}x^{2} + (A^{\vee})_{i, 2}x + (A^{\vee})_{i, 1}) \in \bar{K}[x]$ is also squarefree of degee $5$ or $6$.  Then $A^{\vee} \in U$, so $\Ri(R) = N(A^{\vee})$ is defined, thus proving part (b).  In particular, for any $w \in |T|_{\geq 1}$, $\Ri(\Psi(w))$ is defined and $|\Ri(\Psi(w))| \subset \pp_{\bar{K}}^{1}$ has cardinality $6$.  Part (c) then follows from the same combinatorial argument as was used to prove (a).

\end{proof}

If $R = [(R_{1}, R_{2}, R_{3})] \in \bar{\mathcal{R}}$, let $K(R)$ be the extension of $K$ obtained by adjoining all entries of $M((R_{1}, R_{2}, R_{3}))$.  Note that if two elements of $\mathcal{R}$ are permutation equivalent, their images under $M$ are equivalent up to permutation of the rows.  Thus, $K(R)$ depends only on $R \in \bar{\mathcal{R}}$ and not on the choice of representative of the permutation equivalence class (it is in fact the extension of $K$ fixed by all Galois automorphisms in $\Gal(\bar{K} / K)$ which fix $R$ under the obvious Galois action).

Fix a compatible system $\{\zeta_{2^{n}}\}_{n \geq 0}$ of $2^{n}$-th roots of unity.  For any extension field $L$ of $K$, let $L(\mu_{2}) = \bigcup_{n = 1}^{\infty} L(\zeta_{2^{n}})$.

We write $\rho_{2} : \Gal(\bar{K} / K) \to \Aut_{\zz_{2}}(T_{2}(J))$ for the continuous homomorphism induced by the natural Galois action on $T_{2}(E)$, and denote its image by $G$.  Similarly, for any integer $n \geq 0$, we write $\bar{\rho}_{2}^{(n)} : \Gal(K_{n} / K) \to \GL(J[2^{n}])$ for the homomorphism induced by the natural Galois action on $E[2^{n}]$, and denote its image by $\bar{G}^{(n)}$.  Let $G(n)$ denote the kernel of the natural surjection $G \twoheadrightarrow \bar{G}^{(n)}$; it is the image under $\rho_{2}$ of the normal subgroup $\Gal(\bar{K} / K_{n}) \lhd \Gal(\bar{K} / K)$.  Note that $G(0) = G$.

It follows from the Galois equivariance of the Weil pairing that $G$ is contained in the corresponding group of sympectic similitudes
$$\mathrm{GSp}(T_{2}(J)) := \{\sigma \in \mathrm{Aut}(T_{2}(J))\ |\ e_{2}(P^{\sigma}, Q^{\sigma}) = \chi_{2}(\sigma)e_{2}(P, Q)\ \forall P, Q \in T_{2}(J)\},$$
where $\displaystyle e_{2} : T_{2}(J) \times T_{2}(J) \to \lim_{\leftarrow n} \mu_{2^{n}} \cong \zz_{2}$ is the Weil pairing on the $2$-adic Tate module of $J$ (with respect to the canonical principal polarization of the Jacobian $J$), and $\chi_{2} : G_{K} \to \zz_{2}^{\times}$ is the cyclotomic character on the absolute Galois group of $K$.  Note that, due to the Galois equivariance of the Weil pairing, the image under $\rho_{2}$ of $\Gal(K_{\infty} / K(\mu_{2}))$ coincides with $G \cap \Sp(T_{2}(J))$, where 
$$\Sp(T_{2}(J)) := \{\sigma \in \mathrm{Aut}(T_{2}(J))\ |\ e_{2}(P^{\sigma}, Q^{\sigma}) = e_{2}(P, Q)\ \forall P, Q \in T_{2}(J)\}$$
is the corresponding symplectic group.  The main theorem of \cite{yelton2014images} implies that $G$ contains the principal congruence subgroup $\Gamma(2) \lhd \Sp(T_{2}(J))$ consisting of symplectic automorphisms which are congruent to the identity modulo $2$.  We write $-1 \in \Gamma(2) \subset G$ for the scalar automorphism which acts on $T_{2}(J)$ as multiplication by $-1$.

Now we may state the main result.

\begin{thm}\label{prop: main theorem genus 2}

With notation as above, let $\Psi$ be a decoration on $T$.  Set 

$$K_{\infty}' := K(\{\Psi(v)\}_{v \in |T| \backslash \{v_{0}\}}).$$

a) Choose $i, j \in \{1, 2, 3, 4, 5\}$ with $i \neq j$, and choose an element $\sqrt{\alpha_{i} - \alpha_{j}} \in \bar{K}$ whose square is $\alpha_{i} - \alpha_{j}$.  Then we have 

$$K_{\infty} = K_{\infty}'(\sqrt{\alpha_{i} - \alpha_{j}})(\mu_{2}).$$

b) The Galois automorphism whose image under $\rho_{2}$ is the scalar matrix $-1 \in \Gamma(2)$ acts on $K_{\infty}$ by fixing $K_{\infty}'(\mu_{2})$ and sending $\sqrt{\alpha_{i} - \alpha_{j}}$ to $-\sqrt{\alpha_{i} - \alpha_{j}}$ for $1 \leq i, j \leq 6$, $i \neq j$.

\end{thm}

An outline of this paper is as follows.  In \S2, we define a well-known isogeny of Jacobian surfaces, and in \S3, we give a natural interpretation of the tree $T$ in terms of $\zz_{2}$-lattices contained in $T_{2}(J) \otimes \qq_{2}$.  Finally, \S4 is dedicated to a proof of Theorem \ref{prop: main theorem genus 2} using the machinery established in \S2, 3.

\section{The Richelot isogeny}

We define the Richelot isogeny following \cite{smith2005explicit} (see also \cite{flynn1996prolegomena}).  Let $K$ be as in \S1. The statements of all results in this section assume that the ground field is $K$, although all results remain true for any algebraic extension of $K$.  If $f \in K[x]$ is a squarefree polynomial of degree $5$ or $6$ and $f = G_{1}G_{2}G_{3}$ with each $G_{i} \in K[x]$ a linear or quadratic polynomial, we will refer to the triple $(G_{1}, G_{2}, G_{3})$ as a \textit{quadratic splitting} (of $f = G_{1}G_{2}G_{3}$).  (Note that if $(G_{1}, G_{2}, G_{3})$ is a quadratic splitting, at most one of the $G_{i}$'s is linear.)  For any quadratic splitting $(G_{1}, G_{2}, G_{3})$, where each $G_{i} = G_{i 3}x^{2} + G_{i 2}x + G_{i 1}$ with each $G_{i j} \in K$, let $G$ be the matrix in $M_{3}(\bar{K})$ whose $(i, j)$th entry is $G_{i j}$.

Now for any quadratic splitting $(G_{1}, G_{2}, G_{3})$, for $i = 1, 2, 3$, if $G_{i}$ is quadratic, let $R_{G_{i}}$ be the set of roots of $G_{i}$, and if $G_{i}$ is linear, let $R_{G_{i}} = \{\alpha, \infty\}$ where $\alpha$ is the zero of $G_{i}$.  Then $(R_{G_{1}}, R_{G_{2}}, R_{G_{3}})$ is clearly an element of $\mathcal{R}$.

For two polynomials $G_{1}, G_{2} \in K[x]$, we denote $[G_{1}, G_{2}] = G_{1}G_{2}' - G_{2}G_{1}'$, where the ' symbol indicates the derivative.  If $(G_{1}, G_{2}, G_{3})$ is a quadratic splitting, write $\Ri((G_{1}, G_{2}, G_{3})) = (H_{1}, H_{2}, H_{3})$, where $H_{i} = [G_{i + 1}, G_{i + 2}]$ for $i = 1, 2, 3$ (here we are treating $i$ as an element of $\zz / 3\zz$), and we say that $\Ri((G_{1}, G_{2}, G_{3}))$ is the \textit{Richelot isogenous} quadratic splitting.  The next proposition shows that applying $\Ri$ to a quadratic splitting is compatible to applying $\Ri$ to the corresponding element of $\mathcal{R}$.

\begin{prop}\label{prop: Richelot properties}

Let $(G_{1}, G_{2}, G_{3})$ be a quadratic splitting; let $(H_{1}, H_{2}, H_{3}) = \Ri(G_{1}, G_{2}, G_{3})$; and assume the above notation.  Then we have $(R_{H_{1}}, R_{H_{2}}, R_{H_{3}}) = \Ri((R_{G_{1}}, R_{G_{2}}, R_{G_{3}})) \in \mathcal{R}$.  In particular, $(R_{H_{1}}, R_{H_{2}}, R_{H_{3}})$ is an element of $\mathcal{R}$, or equivalently, $H_{1}H_{2}H_{3}$ is squarefree of degree $5$ or $6$.

\end{prop}

\begin{proof}

Part (a) is straightforward to check through computation, in particular, by proving the identity $H = G^{\vee}$.  The fact that $H_{1}H_{2}H_{3}$ is squarefree of degree $5$ or $6$ follows from part (a) and Proposition \ref{prop: existence of decoration genus 2}(b) and is the statement of Lemma 8.4.2 in \cite{smith2005explicit} in any case.  

\end{proof}

Let $f(x) \in K[x]$ be a polynomial of degree $5$ or $6$ without multiple roots.  Let $C$ be a smooth projective model of the affine hyperelliptic curve over $K$ given by the equation $y^{2} = f(x)$, and let $J$ be its Jacobian.  Let $B \subset \pp_{\bar{K}}^{1}$ be the subset consisting of the $x$-coordinates of the branch points of $C$, where the $x$-coordinate of a point at infinity is $\infty \in \pp_{\bar{K}}^{1}$.  Thus, if $f$ has degree $d$ and has roots $\{\alpha_{i}\}_{i = 1}^{d}$, then $B = \{\alpha_{i}\}_{i = 1}^{d} \cup \{\infty\}$ if $d = 5$ and $B = \{\alpha_{i}\}_{i = 1}^{d}$ if $d = 6$.

Since $J$ is the Jacobian of a curve, it is endowed with a canonical principal polarization.  As in \S1, let $\displaystyle e_{2} : T_{2}(J) \times T_{2}(J) \to \lim_{\leftarrow n} \mu_{2^{n}}$ denote the Weil pairing on $T_{2}(J)$ determined by this polarization.  For each $n \geq 0$, let $\bar{e}_{2}^{(n)} : J[2^{n}] \times J[2^{n}] \to \mu_{2^{n}}$ denote the corresponding Weil pairing on $J[2^{n}]$.  A subgroup $N < J[2^{n}]$ is said to be \textit{Weil isotropic} if it is orthogonal with respect to the Weil pairing $\bar{e}_{2}^{(n)}$; that is, $N$ is Weil isotropic if $\bar{e}_{2}^{(n)}(P, Q) = 1$ for any $P, Q \in N$.  Such a subgroup $N < J[2^{n}]$ is said to be \textit{maximal Weil isotropic} if it is not properly contained in any other Weil isotropic subgroup of $J[2^{n}]$.

We define a correspondence between the maximal Weil isotropic subgroups of $J[2]$ and the permutation equivalence classes $R \in \bar{\mathcal{R}}$  with $|R| = B$ as follows.  It is known (III(a) \S2 Corollary 2.11 of \cite{mumford1984tata}) that each element of $J[2]$ is represented by a divisor of the form $e_{U} := \sum_{\alpha \in U} (\alpha, 0) - \#U \cdot (\infty)$, where $U \subseteq B$ has even cardinality, and conversely, any such divisor $e_{U}$ represents an element of $J[2]$.  Moreover, for two such subsets $U, U' \subseteq B$, $e_{U}$ and $e_{U'}$ are equivalent in $\Pic^{0}(C)$ if and only if $U = U'$ or $U = B \setminus U'$.  Thus, any element of $J[2]$ is represented uniquely by a divisor of the form $e_{U}$, where $U \subset B$ has cardinality $0$ or $2$.  By slight abuse of notation, if $U \subset B$ has cardinality $0$ or $2$, we will consider $e_{U}$ to be an element of $J[2]$.  Note that with this notation, the trivial element of $J[2]$ is $e_{\varnothing}$.  It is also shown in \cite{mumford1984tata}, III(a) \S6 that two elements $e_{U}, e_{U'} \in J[2]$ are Weil isotropic if and only if $U \cap U' = \varnothing$.  Thus, if $N < J[2]$ is a maximal Weil isotropic subgroup, then $N = \langle e_{U}, e_{U'} \rangle$, for some $U, U' \subset B$ each of cardinality $2$ and $U \cap U' = \varnothing$.  Then $N = \{e_{\varnothing}, e_{U}, e_{U'}, e_{U''}\}$, where $U'' = B \setminus (U \cup U')$.  Thus, $[(U, U', U'')] \in \bar{\mathcal{R}}$ with $|(U, U', U'')| = U \cup U' \cup U'' = B$, and the equivalence class $[(U, U', U'')]$ is uniquely determined by $N$.  Conversely, any equivalence class $[(R_{1}, R_{2}, R_{3})] \in \bar{\mathcal{R}}$ with $|[(R_{1}, R_{2}, R_{3})]| = R_{1} \cup R_{2} \cup R_{3} = B$ determines a maximal Weil isotropic subgroup of $J[2]$ given by $\{e_{\varnothing}, e_{R_{1}}, e_{R_{2}}, e_{R_{3}}\}$.  This defines the correspondence.

The following theorem, which states the existence of Richelot isogenies, is proven in \cite{smith2005explicit}, \S8.4.

\begin{thm}\label{prop: Richelot isogeny}

Assuming all of the above notations, let $(G_{1}, G_{2}, G_{3})$ be a quadratic splitting of the polynomial $f(x) \in K[x]$, so that $R_{G} := [(R_{G_{1}}, R_{G_{2}}, R_{G_{3}})] \in \bar{\mathcal{R}}$ with $|R_{G}| = B$.  Let $N$ be the maximal Weil isotropic subgroup of $J[2]$ corresponding to $R_{G}$.  Assume that $\det(G) \neq 0$.  Let $C$ be a smooth projective model of the affine hyperelliptic curve defined over $K$ given by 
\begin{equation}\label{hyperelliptic model genus 2'}
y^{2} = Df(x),
\end{equation}
where $D \in K$, and let $J$ be its Jacobian.  Let $C'$ be a smooth projective model of the affine hyperelliptic curve, defined over $K(R_{G})$, given by 
\begin{equation}\label{Richelot isogenous curve}
y^{2} = \det(G)^{-1}DH_{1}H_{2}H_{3},
\end{equation}
where $(H_{1}, H_{2}, H_{3}) = \Ri(G_{1}, G_{2}, G_{3})$, and let $J'$ be its Jacobian.  Then 

a) there is an isogeny $\psi : J \to J'$ (defined over $K(R_{G})$) whose kernel is $N$; and 

b) the image of $J[2]$ under $\psi$ is the maximal Weil isotropic subgroup of $J'[2]$ corresponding to $\Ri(R_{G})$.

\end{thm}

The isogeny $\psi : J \to J'$ in the above theorem is known as the \textit{Richelot isogeny} corresponding to $N$, and $C'$ (resp. $J'$) is often referred to as the \textit{Richelot isogenous} curve (resp. Jacobian).  This Richelot isogeny $\psi$ is given explicitly in \cite{bending1999curves} \S3 as follows.  Any divisor class in $\Pic^{0}(C)$ can be represented by a divisor of the form $(x, y) - (\alpha, 0)$, with $\alpha \in R_{G_{1}}$.  We set $\psi([(x, y) - (\alpha, 0)]) = [(z_{1}, t_{1}) - (z_{2}, -t_{2})]$, where $z_{1}$ and $z_{2}$ are the roots of the polynomial in $G_{2}(x)H_{2}(z) + G_{3}(x)H_{3}(z) \in F[z]$, and 
\begin{equation}\label{explicit Richelot isogeny}
yt_{i} = \det(G)^{-1}G_{2}(x)H_{2}(z_{i})(x - z_{i})
\end{equation}
for $i = 1, 2$.

\begin{rmk}

If $(G_{1}, G_{2}, G_{3})$ is a quadratic splitting of $f(x)$ and $C$ and $J$ are defined as above, and if $\det(G) = 0$ then $J$ is isogenous to the product of two elliptic curves (see, for instance, Chapter 14 of \cite{flynn1996prolegomena}).  If $(H_{1}, H_{2}, H_{3}) = \Ri((G_{1}, G_{2}, G_{3}))$, then $\det(H) = 2\det(G)^{2}$, so $\det(H) = 0$ also.

\end{rmk}

\section{Description of $T$ via $\zz_{2}$-lattices}

As usual, let 
$$T_{2}(J) = \lim_{\leftarrow n} J[2^{n}]$$
denote the $2$-adic Tate module of $J$; it is a free $\zz_{2}$-module of rank $4$.  Let 
$$V_{2}(J) = T_{2}(J) \otimes \qq_{2}.$$
Then $V_{2}(J)$ is a $4$-dimensional vector space over $\qq_{2}$ which contains the rank-$4$ $\zz_{2}$-lattice $T_{2}(J)$.  Clearly, $\qq_{2}^{\times}$ acts upon the set of all rank-$4$ $\zz_{2}$-lattices in $V_{2}(J)$ as follows: for any such lattice $\Lambda$ and any $a \in \qq_{2}^{\times}$, then $a\Lambda := \{a\lambda \ | \ \lambda \in \Lambda\}$, which is also a rank-$4$ $\zz_{2}$-lattice in $V_{2}(E)$.  Now let $\mathcal{L}$ be the set of equivalence classes of such lattices, where two lattices $\Lambda$ and $\Lambda'$ are equivalent if there exists $a \in \qq_{2}^{\times}$ such that $a\Lambda = \Lambda'$.  The equivalence class in $\mathcal{L}$ of a lattice $\Lambda$ will be denoted $[\Lambda]$.

We will often use ``$<$", ``$\leq$", etc. to indicate inclusion of $\zz_{2}$-lattices inside $V_{2}(J)$, or to indicate inclusion of subgroups of $J[\bar{K}]$.  From now on, let $\Lambda_{0} = T_{2}(J)$.

\begin{prop}\label{prop: description of L}

a) Each element of $\mathcal{L}$ is represented uniquely by a lattice which contains $\Lambda_{0}$ but doesn't contain $\frac{1}{2}\Lambda_{0}$.

b) There is a bijection between elements of $\mathcal{L}$ and finite subgroups of $J[2^{\infty}]$ which do not contain all of $J[2]$.

\end{prop}

\begin{proof}

Let $\Lambda$ be any rank-$4$ $\zz_{2}$-lattice in $V_{2}(J)$.  Consider the sequence of $\zz_{2}$-lattices $\{2^{-n}\Lambda \cap \Lambda_{0}\}_{n = 0}^{\infty}$.  Each lattice $2^{-n}\Lambda \cap \Lambda_{0}$ is an open subgroup of $\Lambda_{0}$ (in the $2$-adic topology); moreover, since $\cup_{n = 0}^{\infty} 2^{-n}\Lambda = V_{2}(J)$, this sequence of lattices forms an open cover of $\Lambda_{0}$.  But $\Lambda_{0}$ is compact, so this cover has a finite subcover.  It follows that for some $N \geq 0$, $2^{-N}\Lambda \geq \Lambda_{0}$.  Now there must be a maximal $M \geq 0$ such that $2^{-N}\Lambda \geq 2^{-M}\Lambda_{0}$, or else $2^{-N}\Lambda \geq \cup_{n = 0}^{\infty} 2^{-n}\Lambda_{0} = V_{2}(J)$, which is impossible.  Then $2^{M - N}\Lambda$ contains $\Lambda$ but not $\frac{1}{2}\Lambda$, and the first part of (a) follows from the fact that $[2^{M - N}\Lambda] = [\Lambda]$.  To prove uniqueness, suppose that $\Lambda'$ and $\Lambda$ are two lattices in the same class which each contain $\Lambda_{0}$ but not $\frac{1}{2}\Lambda_{0}$.  Then $\Lambda' = a\Lambda$ for some $a \in \qq_{2}^{\times}$.  Let $v_{2}(a)$ be the $2$-adic valuation of $a$.  If $v_{2}(a) > 0$, then $a\Lambda$ does not contain $\Lambda_{0}$, which is a contradiction, and if $v_{2}(a) < 0$, then $a\Lambda$ contains $\frac{1}{2}\Lambda_{0}$, which is a contradiction.  Thus, $v_{2}(a) = 0$ and $a \in \zz_{2}^{\times}$, so $\Lambda' = a\Lambda = \Lambda$, and part (a) is proved.

To prove (b), it suffices by (a) to construct a bijection between the set of lattices $\Lambda$ which contain $\Lambda_{0}$ but not $\frac{1}{2}\Lambda_{0}$ and the set of finite subgroups of $J[2^{\infty}]$ which do not contain all of $J[2^{n}]$ for any $n \geq 1$.  Let $\Lambda$ be a lattice containing $\Lambda_{0}$ but not $\frac{1}{2}\Lambda_{0}$.  Then $\Lambda / \Lambda_{0}$ is clearly a finite $\zz_{2}$-module.  Choose $n \geq 0$ such that $2^{n}$ kills $\Lambda / \Lambda_{0}$.  Then clearly $2^{n}\Lambda_{0} \leq 2^{n}\Lambda \leq \Lambda_{0}$, so we may identify $\Lambda / \Lambda_{0} \cong 2^{n}\Lambda / 2^{n}\Lambda_{0}$ with a subgroup of $\Lambda_{0} / 2^{n}\Lambda_{0}$.  But $\Lambda_{0} = T_{2}(J)$, and $T_{2}(J) / 2^{n}T_{2}(J)$ is naturally identified with $J[2^{n}]$, so we may identify $N := \Lambda / \Lambda_{0}$ with a subgroup of $J[2^{n}] < J[2^{\infty}]$.  The fact that $\Lambda$ does not contain $\frac{1}{2}\Lambda_{0}$ implies that $N$ does not contain $J[2]$.  Conversely, let $N$ be a finite subgroup of $J[2^{\infty}]$ which does not contain $J[2]$.  Then 
$N$ is a subgroup of $J[2^{n}]$ for some $n \geq 0$, and furthermore, it has an obvious $\zz_{2}$-module structure.  By the inverse limit definition of $T_{2}(J)$, a subgroup of $J[2^{n}]$ lifts uniquely to a subgroup $M$ of finite index of $T_{2}(J)$ which contains $2^{n}T_{2}(J)$, and such that there is a canonical isomorphism $M / 2^{n}T_{2}(J) \cong N$ which respects the $\zz_{2}$-module structure.  It is clear that $M$ is a sublattice of $T_{2}(J)$.  Let $\Lambda = 2^{-n}M$.  Then there are natural isomorphisms of $\zz_{2}$-modules
\begin{equation}\label{lattice to subgroup} \Lambda / T_{2}(J) \cong 2^{n}\Lambda / 2^{n}T_{2}(J) \cong M / 2^{n}T_{2}(J) \cong N. \end{equation}
By construction, this $\Lambda$ is uniquely determined by $N$.  Moreover, the fact that $N$ does not contain $J[2]$ implies that $\Lambda$ does not contain $\frac{1}{2}\Lambda_{0}$.
\end{proof}

We define a graph $S$ whose vertices form a subset of $\mathcal{L}$ as follows.  Let the set of vertices $|S|$ be all elements of $\Lambda$ corresponding under the bijection given by Proposition \ref{prop: description of L} to maximal Weil isotropic subgroups $N < J[2^{n}]$ which do not contain $J[2]$ for some $n \geq 0$.  Two vertices in $|S|$ are connected by an edge if they can be written as $[\Lambda]$ and $[\Lambda']$, where $\Lambda < \Lambda'$ and $\Lambda' / \Lambda \cong \zz /2\zz \oplus \zz / 2\zz$.  It is easy to see that this relation is symmetric, so the edge set of this graph is well-defined.  Set $v_{0} = [\Lambda_{0}] \in |S|$.

\begin{lemma}\label{prop: size of maximal isotropic subgroup}

Let $M$ be a free module of (finite) even rank $2m$ over a finite commutative ring $\Sigma$, equipped with a nondegenerate alternating bilinear pairing $\langle \cdot, \cdot \rangle : M \times M \to \Sigma$.  Then a maximal isotropic submodule of $M$ is maximal isotropic if and only if it has cardinality $|\Sigma|^{m}$.

\end{lemma}

\begin{proof}

Let $N$ be an isotropic submodule of $M$, and let $N^{\bot} = \{m \in M \ | \ \langle m, n \rangle = 0 \ \forall n \in N\}$.  Since $N$ is isotropic, $N \subseteq N^{\bot}$.  Clearly, $N$ is a maximal isotropic submodule if and only if $N = N^{\bot}$; thus, it suffices to show that $|N^{\bot}| = |M| / |N| = |\Sigma|^{2m} / |N|$.

Define a map from $M$ to its dual space $M^{*}$ given by $m \mapsto \langle \cdot, m \rangle$.  Since $\langle \cdot, \cdot \rangle$ is nondegenerate, this is an isomorphism.  Moreover, it maps $N^{\bot}$ to the submodule $\{\alpha \in M^{*} \ | \ \alpha(N) = 0\} \cong (M / N)^{*}$, which has cardinality $|M / N|$.  Thus, $N^{\bot}$ has cardinality equal to $|M| / |N|$, as desired.

\end{proof}

\begin{lemma}\label{prop: 15 subgroups}

Let $M$ be a vector space of dimension $4$ over $\ff_{2}$, equipped with a nondegenerate alternating bilinear pairing $\langle \cdot, \cdot \rangle : M \times M \to \ff_{2}$.  Then there are exactly $15$ maximal isotropic subspaces of $M$.

\end{lemma}

\begin{proof}

This is well known and can be proved using elementary methods, for instance, choosing a symplectic basis of $M$ and checking all subspaces of dimension $2$.  Alternately, one can prove the lemma in the case of $M = J[2]$ by noting that, as in the proof of Proposition \ref{prop: existence of decoration genus 2}(a), there are exactly $15$ permutation equivalence class representatives $R \in \mathcal{R}$ with $|R| = B$, which are in one-to-one correspondence with the maximal Weil isotropic subgroups of $J[2]$ as in \S2.

\end{proof}

The following is well known (see, for instance, \cite{milne1986abelian}, Chapter V, Theorem 16.4 combined with Proposition 16.8).

\begin{lemma}\label{prop: Weil pairing under isogeny}

Let $J$ be the Jacobian of a curve, whose principal polarization determines a Weil pairing $\displaystyle e_{2} : T_{2}(J) \times T_{2}(J) \to \lim_{\leftarrow n} \mu_{2^{n}}$.  Let $J'$ be another abelian variety, and let $\phi : J \to J'$ be an isogeny whose kernel is a maximal Weil isotropic subgroup of $J[2]$.  Then $J'$ has a principal polarization which determines a Weil pairing $\displaystyle e_{2}' : T_{2}(J') \times T_{2}(J') \to \lim_{\leftarrow n} \mu_{2^{n}}$, and $e_{2}'(\phi(P), \phi(Q)) = e_{2}(P, Q)^{2}$ for all $P, Q \in J[2^{\infty}]$.

\end{lemma}

\begin{prop}\label{prop: S is 15-regular}

With the above definition, $S$ is a connected, $15$-regular graph.

\end{prop}

\begin{proof}

Below, for $n \geq 1$ and any subgroup of $N < J[2^{n}]$, considered as a $\zz / 2^{n}\zz$-module, the \textit{rank} of $N$ is defined to be the dimension of $N \otimes_{\zz / 2^{n}\zz} \ff_{2}$ as a vector space over $\ff_{2}$.

We first show that $S$ is connected by proving that for any vertex $v \in |S|$, there is a path from $v$ to the root $v_{0}$.  This is trivially true for $v = v_{0}$.  For any vertex $v \in |S|$, let $N_{v}$ be the subgroup corresponding to it in the definition of $S$, and let $m(v)$ be the (unique) integer such that $N_{v} < J[2^{m(v)}]$ is a maximal Weil isotropic subgroup not containing $J[2]$.  Choose any vertex $v \neq v_{0}$, and assume inductively that the claim holds for any $n \leq m(v) - 1$.  Since $J[2^{m(v)}]$ is a free $\zz / 2^{m(v)}\zz$-module of rank $4$, $N_{v}$ may be viewed as a $\zz / 2^{m(v)}\zz$-module of rank $\leq 4$.  By Lemma \ref{prop: size of maximal isotropic subgroup}, the order of $N_{v}$ is $2^{2n}$, which forces its rank to be $\geq 2$.  If the rank of $N_{v}$ were equal to $4$, $N_{v}$ would contain $J[2]$, so the rank of $N_{v}$ is $2$ or $3$.  If the rank is $2$, then the condition on the order of $N_{v}$ forces $N_{v} \cong (\zz / 2^{m(v)}\zz)^{2}$.  In this case, let $N'$ be the (unique) submodule of $N_{v}$ isomoprhic to $(\zz / 2^{m(v) - 1}\zz)^{2}$.  Otherwise, $N_{v}$ is of rank $3$.  In this case, $N_{v}$ must contain an element of order $2^{m(v)}$; otherwise $\langle x, y \rangle = 0$ for all $x \in J[2]$ and all $y \in N_{v}$, and by maximality, $N_{v} \supset J[2]$.  Let $y$ be an element in $N_{v}$ of order $2^{m(v)}$.  Then $N_{v} \cong y \cdot \zz / 2^{m(v)}\zz \oplus M$, where $M$ is a submodule of rank $2$ and order $2^{m(v)}$.  Let $x$ be an element of maximal order in $M$, and let $N'$ be the submodule of $N_{v}$ generated by $2x$ and $2y$.  In either case $N' < J[2^{m(v) - 1}]$, and Lemma \ref{prop: Weil pairing under isogeny} implies that $N'$ is Weil isotropic in $J[2^{m(v) - 1}]$.  Moreover, in either case, $N'$ has order $2^{2(m(v) - 1)}$, so Lemma \ref{prop: size of maximal isotropic subgroup} implies that $N'$ is a maximal Weil isotropic subgroup of $J[2^{m(v) - 1}]$ and represents a vertex $v' \in |S|$.  Since in either case, $N' < N_{v}$ and $N_{v} / N' \cong \zz / 2\zz \oplus \zz / 2\zz$, $v'$ is connected to $v$ by an edge.  Then by the inductive assumption, there is a path from $v'$ to $v_{0}$, and so there is also a path from $v$ to $v_{0}$.  Thus, $S$ is connected.

We now prove that $S$ is $15$-regular.  First of all, Lemma \ref{prop: 15 subgroups} implies that there are exactly $15$ maximal Weil isotropic subgroups of $J[2]$, which implies that there are exactly $15$ vertices adjacent to $v_{0}$.  Now let $v$ be any vertex different from $v_{0}$, and let $N_{v}$ be the corresponding maximal Weil isotropic subgroup of $J[2^{m(v)}]$ not containing $J[2]$.  Lemmas \ref{prop: size of maximal isotropic subgroup} and \ref{prop: Weil pairing under isogeny} imply that a subgroup $N' < J[2^{m(v) + 1}]$ containing $N_{v}$ is maximal Weil isotropic if and only if $N' / N_{v}$ is a maximal Weil isotropic subgroup of $(J / N_{v}) [2]$.  Since Lemma \ref{prop: 15 subgroups} implies that there are exactly $15$ maximal Weil isotropic subgroups of $(J / N_{v}) [2]$, it follows that there are exactly $15$ maximal Weil isotropic subgroups of $J[2^{m(v) + 1}]$ which contain $N_{v}$.

It will now suffice to produce a bijection between the vertices adjacent to $v$ and the $15$ maximal Weil isotropic subgroups of $J[2^{m(v) + 1}]$ containing $N_{v}$.  Let $\Lambda_{v}$ be the free rank-$4$ $\zz_{2}$-lattice in $V_{2}(J)$ corresponding to $v$ via Proposition \ref{prop: description of L}(a).  Every vertex of $S$ is represented uniquely by a lattice which contains $\Lambda_{v}$ but not $\frac{1}{2}\Lambda_{v}$, by an argument similar to that used in the proof of Proposition \ref{prop: description of L}(a).  Suppose that $v'$ is a vertex adjacent to $v$, and that $\Lambda'$ and $\Lambda$ are lattices representing $v'$ and $v$ respectively such that one lattice is contained in the other with quotient isomorphic to $\zz / 2\zz \oplus \zz / 2\zz$.  Then, after possibly replacing $\Lambda'$ with $\frac{1}{2}\Lambda'$ and then possibly multiplying both lattices by a suitable scalar, one can assume that $\Lambda = \Lambda_{v}$ and $\Lambda' > \Lambda_{v}$ with $\Lambda' / \Lambda_{v} \cong \zz / 2\zz \oplus \zz / 2\zz$.  The latter condition implies that $\Lambda'$ does not contain $\frac{1}{2}\Lambda_{v}$.  Thus, each vertex $v'$ adjacent to $v$ is uniquely represented by a lattice $\Lambda' > \Lambda_{v}$ with $\Lambda' / \Lambda \cong \zz / 2\zz \oplus \zz / 2\zz$.  Now as in the proof of Proposition \ref{prop: description of L}(b), $N_{v}$ is the subgroup of $J[2^{\infty}]$ not containing $J[2]$ corresponding to $\Lambda_{v}$, and $N_{v} \cong \Lambda_{v} / \Lambda_{0}$.  For any such lattice $\Lambda'$, clearly $2^{m(v) + 1}\Lambda' < \Lambda_{0}$, so we may identify $\Lambda' / \Lambda_{0} \cong 2^{m(v) + 1}\Lambda' / 2^{m(v) + 1}\Lambda_{0}$ with a subgroup of $\Lambda_{0} / 2^{m(v) + 1}\Lambda_{0}$, which is naturally identified with $J[2^{m(v) + 1}]$.  Thus, we may identify $N' := \Lambda' / \Lambda_{0}$ with a subgroup of $J[2^{m(v) + 1}]$.  Let $m(v')$ be the minimal positive integer such that $2^{m(v')}$ kills $\Lambda' / \Lambda_{0}$; clearly, $m(v') \leq m(v) + 1$.  Then $N_{v'} := 2^{m(v')}\Lambda' / 2^{m(v')}\Lambda_{0} < \Lambda_{0} / 2^{m(v')}\Lambda_{0}$ is the subgroup of $J[2^{m(v')}]$ associated to $\Lambda'$ in the proof of Proposition \ref{prop: description of L}(b).  By the definition of $S$, $N_{v'}$ is a maximal Weil isotropic subgroup of $J[2^{m(v')}]$.  But clearly $N_{v'} = 2^{m(v) + 1 - m(v')}N'$, so $N'$ is Weil isotropic.  By checking the order of $N'$ and applying Lemma \ref{prop: size of maximal isotropic subgroup}, we verify that $N'$ is a maximal Weil isotropic subgroup of $J[2^{m(v) + 1}]$; moreover, $\Lambda' > \Lambda_{v}$ implies $N' > N_{v}$.  

Now assume conversely that $N'$ is a maximal Weil isotropic subgroup of $J[2^{m(v) + 1}]$.  Then, by the construction used in the proof of Proposition \ref{prop: size of maximal isotropic subgroup}(b), $N'$ corresponds to a lattice $\Lambda'$ containing $\Lambda_{0}$ but not $\frac{1}{2}\Lambda_{0}$, and $\Lambda' > \Lambda_{v}$ with $\Lambda' / \Lambda_{v} \cong N' / N_{v}$.  Since $N'$ has order $2^{2(m(v) + 1)}$ and $N_{v}$ has order $2^{2m(v)}$, the quotient $N' / N_{v}$ has order $4$.  As was shown above, $N_{v}$ has an element of order $2^{m(v)}$, so the elements of $N' / N_{v}$ must all have order dividing $2$.  It follows that $\Lambda' / \Lambda_{v} \cong N' / N_{v} \cong \zz / 2\zz \oplus \zz / 2\zz$.  Then by the definition of $S$, the vertex $v' \in |S|$ represented by $\Lambda'$ is adjacent to $v$, and we are done.

\end{proof}

Instead of working with $S$, we want to work with a graph that has nicer properties (for instance, $S$ is not simply connected).  Since $S$ is connected and $15$-regular with basepoint $v_{0}$, it has a universal covering graph which is a $15$-regular tree, and we may identify it with $T$ from \S1.  Thus, each vertex $w \in |T|$ corresponds to a non-backtracking path in $S$ beginning at $v_{0}$, which we write as a sequence of vertices $\{v_{0}, v_{1}, ... , v_{n}\}$ with each $v_{i} \in |S|$ and $v_{i - 1}$ and $v_{i}$ adjacent for $1 \leq i \leq n$.  We designate $w_{0} := \{v_{0}\}$ as the root of the tree $T$.

\begin{prop}\label{prop: proporties of w mapsto N_{w}}

For any $v \in |S|$, let $N_{v}$ be the maximal Weil isotropic subgroup of $J[2^{m}]$ for some $m$ not containing $J[2]$ which uniquely corresponds to $v$ as in the definition of $S$.  For any $w = \{v_{0}, ... , v_{n}\} \in |T|$, let $m(w)$ be the unique integer such that $N_{v_{n}} < J[2^{m(w)}]$ is a maximal Weil isotropic subgroup.  Then,

a) $m(v) \leq n$ and $n - m(v)$ is even.

Define $N_{w} = 2^{(m(v) - n)/2}N_{v_{n}}$.  Then the assignment $w \mapsto N_{w}$ has the following properties.

b) If $w \in |T|_{n}$, then $N_{w}$ is a maximal Weil isotropic subgroup of $J[2^{n}]$; and 

c) If $w \in |T|_{n}$ for $n \geq 1$, and $\tilde{w} \in |T|_{n - 1}$ is its parent vertex, then $N_{w} > N_{\tilde{w}}$ and $N_{w} / N_{\tilde{w}} \cong \zz / 2\zz \oplus \zz / 2\zz$.

\end{prop}

\begin{proof}

If $w \in |T|_{1}$, then $w = \{v_{0}, v_{1}\}$ and $N_{v_{1}}$ is a maximal Weil isotropic subgroup of $J[2]$, so $m(w) = 1$ and parts (a) and (c) are clear.  Thus, all the claims are proven for $n = 1$ (note that (a) is trivially true for $n = 0$).  Choose $n \geq 2$, and assume inductively that all the claims are true for $n - 1$.  Choose $w = \{v_{0}, v_{1}, ... , v_{n}\} \in |T|_{n}$; then we may apply the inductive assumptions to $\tilde{w} = \{v_{0}, v_{1}, ... , v_{n - 1}\} \in |T|_{n - 1}$.

It was shown in the proof of Proposition \ref{prop: S is 15-regular} that $m(w) = m(\tilde{w}) + 1$ or $m(w) = m(\tilde{w}) - 1$.  In the first case, $n - m(w) = (n - 1) - m(\tilde{w})$, which is nonnegative and even, so $n - m(w)$ is nonnegative and even, which is the claim of part (a).  Moreover, as in the proof of Proposition \ref{prop: S is 15-regular}, $N_{v_{n}} > N_{v_{n - 1}}$ with $N_{v_{n}} / N_{v_{n - 1}} \cong \zz / 2\zz \oplus \zz / 2\zz$.  But $N_{w} = 2^{(m(w) - n)/2}N_{v_{n}}$ and $N_{\tilde{w}} = 2^{(m(w) - n)/2}N_{v_{n - 1}}$, and it follows immediately that $N_{w} > N_{\tilde{w}}$ with $N_{w} / N_{\tilde{w}} \cong \zz / 2\zz \oplus \zz / 2\zz$, which is the claim of part (c).  Now in the second case that $m(w) = m(\tilde{w}) - 1$, $n - m(w) = (n - 1) - m(\tilde{w}) + 2$.  Since $(n - 1) - m(\tilde{w})$ is nonnegative and even, so is $n - m(w)$, which is the claim of part (a).  Moreover, as in the proof of Proposition \ref{prop: S is 15-regular}, $N_{v_{n}} < N_{v_{n - 1}}$ and $N_{v_{n - 1}} / N_{v_{n}} \cong \zz / 2\zz \oplus \zz / 2\zz$.  Then $N_{v_{n}} > 2N_{v_{n - 1}}$ with $N_{v_{n}} / 2N_{v_{n - 1}} \cong \zz / 2\zz \oplus \zz / 2\zz$.  But $N_{w} = 2^{(m(w) - n)/2}N_{v_{n}}$ and $N_{\tilde{w}} = 2^{(m(w) - n)/2} \cdot 2N_{v_{n - 1}}$, and it follows immediately again that $N_{w} > N_{\tilde{w}}$ with $N_{w} / N_{\tilde{w}} \cong \zz / 2\zz \oplus \zz / 2\zz$, which is the claim of part (c).  Thus, (a) and (c) are proven.

Let $P, Q \in N_{w}$.  Then $2^{n - m(w)}P, 2^{n - m(w)}Q \in N_{v_{n}}$ and thus, $\bar{e}_{2}^{(n)}(P, Q) = \bar{e}_{2}^{(m(w))}(2^{n - m(w)}P, 2^{n - m(w)}Q) = 1$.  Thus, $N_{w}$ is a Weil isotropic subgroup of $J[2^{n}]$.  Since $N_{v_{n}} < J[2^{m(w)}]$ is maximal Weil isotropic, by Lemma \ref{prop: size of maximal isotropic subgroup}, $|N_{v_{n}}| = 2^{2m(w)}$.  Then 
\begin{equation}\label{proving properties of w mapsto N_{w}}
|N_{w}| = |N_{v_{n}}| \cdot |J[2^{n - m(w)/2}]| = 2^{2m(w)} \cdot 2^{2(n - m(w))} = 2^{2n}.
\end{equation}
Now Lemma \ref{prop: size of maximal isotropic subgroup} implies that $N_{w}$ is maximal Weil isotropic in $J[2^{n}]$, thus proving part (b).

\end{proof}

\begin{rmk}

Note that for general $w, w' \in |T|$, $N_{w}$ may contain $J[2]$, and that $N_{w} = N_{w'}$ does not imply that $w = w'$.

\end{rmk}

\section{Proof of main result}

The goal of this section is to prove Theorem \ref{prop: main theorem genus 2}.  We retain the notations of previous sections.  In particular, $C$ is a smooth, projective model of the affine hyperelliptic curve given by (\ref{hyperelliptic model genus 2}), which is defined over $K$, and $J$ is its Jacobian.  We also define the tree $T$ as in \S3, and for each $w \in |T|_{n}$, we define the maximal Weil isotropic subgroup $N_{w} < J[2^{n}]$ as in Proposition \ref{prop: proporties of w mapsto N_{w}}.

We now assign to each $w \in |T|_{n}$ a Jacobian surface $J_{w}$ and an isogeny $\phi_{w} : J \to J_{w}$ whose kernel is $N_{v_{n}}$, which we will later show is defined over $K(N_{v_{n}})$.

Set $J_{w_{0}} := J$, and let $\phi_{w_{0}} : J \to J_{w_{0}}$ be the identity isogeny.

For any $w \in |T|_{1}$, let $C_{w}$ (resp. $J_{w}$) be the Richelot isogenous curve (resp. Jacobian) corresponding to $N_{w}$ as in Theorem \ref{prop: Richelot isogeny}, and let $\phi_{w} : J \to J_{w}$ be the corresponding Richelot isogeny.  Note that, as in \S3, the maximal Weil isotropic subgroup $N_{w} < J[2]$ determines a permutation equivalence class $R_{w} \in \bar{\mathcal{R}}$ with $|R_{w}| = \{\alpha_{i}\}_{i = 1}^{5} \cup \{\infty\}$, and that $|\Ri(R_{w})|$ is the set of $x$-coordinates of branch points of $C_{w}$.  Moreover, $\phi_{w}$ and $J_{w}$ are defined over $K(R_{w})$.

Now choose $w \in |T|_{n}$ for some $n \geq 1$, and assume inductively that a curve $C_{w}$ whose Jacobian is $J_{w}$, as well as an isogeny $\phi_{w} : J \to J_{w}$ whose kernel is $N_{w}$, have been defined.  Assume further that we have defined a curve $C_{\tilde{w}}$ whose Jacobian is $J_{\tilde{w}}$ and an isogeny $\phi_{\tilde{w}} : J \to J_{\tilde{w}}$ whose kernel is $N_{\tilde{w}}$.  Moreover, assume that we have assigned an element $R_{w} \in \mathcal{R}$ such that $|R_{w}|$ is the set of $x$-coordinates of branch points of $C_{\tilde{w}}$.  Let $u \in |T|$ with $w = \tilde{u}$.  Then there is a corresponding maximal Weil isotropic subgroup $N_{u} < J[2^{n + 1}]$ containing $N_{w}$.  So $\phi_{w}(N_{u})$ is a maximal Weil isotropic subgroup of $J_{w}[2]$, which again determines a permutation equivalence class $R_{u} \in \bar{\mathcal{R}}$ such that $|R_{u}|$ is the set of $x$-coordinates of branch points of $C_{w}$.  Thus, $|R_{u}| = |\Ri(R_{w})|$.  Now let $C_{u}$ (resp. $J_{u}$) be the Richelot isogenous curve (resp. Jacobian) and $\psi_{u} : J_{w} \to J_{u}$ be the Richelot isogeny associated to the maximal Weil isotropic subgroup $\phi_{w}(N_{u}) < J_{w}[2]$ as in Theorem \ref{prop: Richelot isogeny}.  Then $\phi_{u} := \psi_{u} \circ \phi_{w} : J \to J_{u}$ is an isogeny whose kernel is $N_{u}$.  Since the parent of every vertex in $|T|_{n + 1}$ is a vertex in $|T|_{n}$, it follows that through the method described above, we have defined the desired $C_{w}$, $J_{w}$ and $\phi_{w}$ for all $w \in |T|_{n + 1}$.

In this way, $C_{w}$, $J_{w}$, $\phi_{w}$, and $R_{w} \in \mathcal{R}$ are defined for all $w \in |T|_{\geq 1}$.  Furthermore, for all $w \in |T|$, we define $K_{w}$ to be the extension of $K$ obtained by adjoining the coefficients of the Weierstrass equation of $C_{w}$ given above.

\begin{lemma}\label{prop: Psi is a decoration genus 2}

Using the above notation, define $\Psi : |T|_{\geq 1} \to \bar{\mathcal{R}}$ by setting $\Psi(w) = R_{w}$ for $w \in |T|_{\geq 1}$.  Then $\Psi$ is a decoration on $T$.

\end{lemma}

\begin{proof}

We have to show that the conditions in Definition \ref{def: decoration genus 2} are fulfilled.  First, let $w, w' \in |T|_{n}$ for some $n \geq 1$, and assume they have the same parent vertex.  So there are maximal Weil isotropic subgroups $N_{w}, N_{w'} < J[2^{n}]$, $N_{\tilde{w}} < J[2^{n - 1}]$, such that $N_{w}, N_{w'} > N_{\tilde{w}}$.  Then $\phi_{\tilde{w}}(N_{w})$ and $\phi_{\tilde{w}}(N_{w'})$ are distinct maximal Weil isotropic subgroups of $J_{\tilde{w}}[2]$, and it follows that the associated elements $\Psi(w) = R_{w}, \Psi(w') = R_{w'} \in \bar{\mathcal{R}}$ must be distinct, thus satisfying part (a).

It is clear from the construction of $R_{w}$ for $w \in |T|_{1}$ that part (b) is satisfied.

Finally, let $u = \{v_{0}, ... , v_{n - 1}, v_{n}, v_{n + 1}\} \in |T|_{n + 1}$ and $w = \tilde{u} = \{v_{0}, ... , v_{n - 1}, v_{n}\} \in |T|_{n}$.  Then by the above construction, $\phi_{w} = \psi_{w} \circ \phi_{\tilde{w}}$, where $\psi_{w} : J_{\tilde{w}} \to J_{w}$ is the Richelot isogeny associated to the maximal Weil isotropic subgroup $\phi_{\tilde{w}}(N_{w}) < J_{w}[2]$.  Now suppose that $\Psi(u) = R_{u}$ coincides with $\Ri(\Psi(w)) = \Ri(R_{w})$.  Then Theorem \ref{prop: Richelot isogeny}(b) says that $\phi_{w}(N_{u})$ is the image of $J_{\tilde{w}}[2]$ under $\psi_{w}$.  It follows that $N_{\tilde{w}} = 2N_{u}$.  Then by the construction in Proposition \ref{prop: proporties of w mapsto N_{w}}, $N_{v_{n - 1}} = N_{v_{n + 1}}$, so $v_{n + 1} = v_{n - 1}$, which contradicts the fact that the path $u = \{v_{0}, ... , v_{n - 1}, v_{n}, v_{n + 1}\} \in |T|$ is non-backtracking.  Thus, $\Psi(u) \neq \Ri(\Psi(w))$, and part (c) is satisfied.

\end{proof}

\begin{dfn}

For any integer $n \geq 0$, define the extension $K_{n}'$ of $K$ to be the compositum of the fields $K_{w}$ for all $w \in |T|_{\geq 1}$.  Define the extension $K_{\infty}'$ of $\bar{K}$ to be the infinite compositum 
$$K_{\infty}' := \bigcup_{n \geq 0} K_{n}'.$$
\end{dfn}

In this way, we obtain a tower of field extensions
\begin{equation}\label{tower of fields genus 2} K = K_{0}' \subset K_{1}' \subset K_{2}' \subset ... \subset K_{n}' \subset ..., \end{equation}
with $K_{\infty}' = \bigcup_{n \geq 0} K_{n}'$.

\begin{lemma}\label{prop: field of definition genus 2}

For any $w \in |T|_{n}$, let $\{w_{0}, w_{1}, ... , w_{n}\}$ be the sequence of vertices in the path of length $n$ from $w_{0}$ to $w$.  Let $\tilde{K}_{w}$ denote the compositum of the fields $K_{w}$ for all $w \in \{w_{0}, w_{1}, ... , w_{n}\}$.  Then 
$$\tilde{K}_{w} = K(\{R_{w}\}_{w \in \{w_{1}, w_{2}, ... , w_{n}\}}).$$

\end{lemma}

\begin{proof}

This is trivial for $n = 0$.  Now assume inductively that the statement holds for some $n \geq 0$ and all $w \in |T|_{n}$.  Choose any $w \in |T|_{n + 1}$.  We may apply the inductive assumption to $\tilde{w}$, since $\tilde{w} \in |T|_{n}$.  One checks from the form of (\ref{Richelot isogenous curve}) that the curve $C_{w}$, and hence also its Jacobian $J_{w}$, are defined over $\tilde{K}_{\tilde{w}}(R_{w})$.

It will now suffice to show that any field over which $J_{w}$ is defined must contain $\tilde{K}_{\tilde{w}}(R_{w})$, and it will follow that $\tilde{K}_{w} = \tilde{K}_{\tilde{w}}(R_{w})$.  To do this, recall that $\phi_{w} : J \to J_{w}$ is the composition of $\phi_{\tilde{w}} : J \to J_{\tilde{w}}$ with a Richelot isogeny $\psi_{w} : J_{\tilde{w}} \to J_{w}$ whose kernel is the maximal Weil isotropic subgroup $N := \phi_{\tilde{w}}(N_{w}) < J_{\tilde{w}}[2]$.  Let $\sigma$ be an automorphism of the field $\tilde{K}_{\tilde{w}}(J_{\tilde{w}}[2])$ fixing $\tilde{K}_{\tilde{w}}$, and suppose that $\sigma$ fixes $\psi_{w}$ and $J_{w}$.  Since $\psi_{w}^{\sigma} : J_{\tilde{w}} \to J_{w}^{\sigma}$ is an isogeny whose kernel is $N^{\sigma}$, it follows that $\sigma$ stabilizes $N$ as well.  Therefore, any field over which $\psi_{w}$ and $J_{w}$ are defined must contain the subfield $L \subset \tilde{K}_{\tilde{w}}(J_{\tilde{w}}[2])$ fixed by all such automorphisms $\sigma$ which stabilize $N$.  Recall that the $4$ elements of $N$ are represented by divisors of the form $e_{\varnothing}, e_{R_{1}}, e_{R_{2}}, e_{R_{3}} \in \Div^{0}(J_{\tilde{w}})$, where $R_{w} = [(R_{1}, R_{2}, R_{3})]$.  Thus, $L$ is the field fixed by all automorphisms $\sigma$ which fix $(R_{1}, R_{2}, R_{3})$.  It is easy to check from the construction of $M : \mathcal{R} \to M_{3}(\bar{K})$ that this field is generated by the entries of $M(R_{w})$, and so $L = \tilde{K}_{\tilde{w}}(R_{w})$, as desired.  

\end{proof}

\begin{prop}\label{prop: characterization of fields' genus 2}

a) For any $n \geq 1$, $K_{n}' = K(\{\Psi(w)\}_{w \in |T|_{\leq n} \backslash \{w_{0}\}})$ for any decoration $\Psi$ on $T$.

b) As in the statement of Theorem \ref{prop: main theorem genus 2}, $K_{\infty}' = K(\{\Psi(w)\}_{w \in |T|_{\geq 1}})$ for any decoration $\Psi$ on $T$.

(In particular, the extensions $K(\{\Psi(w)\}_{w \in |T|_{\leq n} \backslash \{w_{0}\}})$ and $K(\{\Psi(w)\}_{w \in |T|_{\geq 1}})$ do not depend on the choice of decoration $\Psi$.)

\end{prop}

\begin{proof}

It follows directly from the definition of $K_{n}'$ and the statement of Lemma \ref{prop: field of definition genus 2} that 
\begin{equation}\label{characterization of fields'} K_{n}' = K(\{R_{w}\})_{w \in |T|_{\leq n} \backslash \{w_{0}\}}), \end{equation}
from which it follows that
\begin{equation} K_{\infty}' = K(\{R_{w}\}_{w \in |T| \backslash \{w_{0}\}}). \end{equation}
Therefore, it suffices to show that for any decoration $\Psi$, $K(\{\Psi(w)\}_{w \in |T|_{n} \backslash \{w_{0}\}}) = K(\{R_{w}\}_{w \in |T|_{n} \backslash \{w_{0}\}}).$  Choose any decoration $\Psi$.  By Definition \ref{def: decoration genus 2} parts (a) and (b), the $15$ elements $\Psi(v)$ for $v \in |T|_{1}$ are representatives of all $15$ permutation equivalence classes of elements $R \in \mathcal{R}$ such that $|R| = \{\alpha_{i}\}_{i = 1}^{5} \cup \{\infty\}$.  It follows that there is a permutation $\sigma$ on $|T|_{1}$ such that for each $v \in |T|_{1}$, $\Psi(v) = R_{\sigma(v)}$.  Now assume inductively that for some $n \geq 1$, there is a permutation $\sigma$ on $|T|_{\leq n}$ which preserves distances between vertices (in particular, it acts on each $|T|_{i}$ for $1 \leq i \leq n$), such that $\Psi(w) = R_{\sigma(w)}$ for all $w \in |T|_{\leq n}$.  Now choose any $w \in |T|_{n}$.  By Definition \ref{def: decoration genus 2} parts (a) and (c), the $14$ elements $\Psi(u)$ for any $u$ such that $w = \tilde{u}$ coincide with the $14$ elements $R \in \bar{\mathcal{R}}$ such that $|R| = |\Ri(R_{w})|$.  It follows from the inductive assumption that for each such $u$, there is a unique $u'$ whose parent is $\sigma(w)$ such that $R_{w} = \Psi(u')$.  Extend $\sigma$ to be a permutation on $|T|_{n + 1}$ by assigning $\sigma(u) = u'$.  Since every vertex in $|T|_{n + 1}$ has its parent in $|T|_{n}$, it is clear that $\sigma$ is defined on $|T|_{n + 1}$; moreover, one can easily check that $\sigma$ is still a permutation which preserves distances between vertices.  Thus, we have the equalities 
$$K(\{\Psi(w)\}_{w \in |T|_{\leq n + 1} \backslash \{v_{0}\}}) = K(\{R_{\sigma(w)})\}_{w \in |T|_{\leq n + 1}} \backslash \{w_{0}\})$$
\begin{equation}\label{proving characterization of fields' genus 2} 
 = K(\{R_{w}\}_{w \in |T|_{\leq n + 1}} \backslash \{w_{0}\}), \end{equation}
and we are done.

\end{proof}

\begin{prop}\label{prop: field of definition of isogenies genus 2} With the above notation, 

a) the isogeny $\phi_{w}$ is defined over $K(N_{w})$, and $K_{w} \subseteq K(N_{w})$,

b) for all $n \geq 0$, $K_{n}' \subseteq K_{n}$, and equality holds for $n = 0, 1$.

\end{prop}

\begin{proof}

First of all, we have shown in the proof of \ref{prop: field of definition genus 2} that $\phi_{w}$ is defined over $K(N_{w})$ for $w \in |T|_{1}$, and that in this case, $K_{w} = \tilde{K}_{w}$ is the subfield of $K_{1}$ fixed by all automorphisms $\sigma \in \Gal(K_{1} / K)$ which stabilize $N_{w}$.  It follows that $K_{w} \subseteq K(N_{w})$, and part (a) is proven in the case that $n = 1$.  Now by Proposition \ref{prop: characterization of fields' genus 2}(a), $K_{1}' = K(\{R_{w}\}_{w \in |T|_{1}})$.  The fact that this is contained in $K(\{\alpha_{i}\}_{i = 1}^{6}) = K_{1}$ follows immediately from the fact that $|R_{w}| = \{\alpha_{i}\}_{i = 1}^{6}$ for each $w \in |T|_{1}$.  Since $K_{1}'$ is the compositum of all such subfields $K(R_{w})$, and each $K(R_{w})$ is the subfield of $K_{1}$ fixed by all automorphisms $\sigma \in \Gal(K_{1} / K)$ which stabilize $N_{w}$, it follows that $K_{1}'$ is the subfield of $K_{1}$ fixed by the elements of $\Gal(K_{1} / K)$ which stabilize all maximal Weil isotropic subgroups $N_{w} < J[2]$.  But the only such Galois element is the identity, so $K_{1}' = K_{1}$.  This proves equality in the $n = 1$ case of the statement of part (b) (the equality in the $n = 0$ case is trivial).  Thus, all claims are proven for $n = 1$.

Now assume inductively that for some $n \geq 1$ and all $w \in |T|_{n}$, $\phi_{w}$ is defined over $K(N_{w})$ and $K_{w} \subseteq K(N_{w})$.  Choose any $w \in |T|_{n + 1}$.  We may apply the inductive assumption to $\tilde{w}$, since $\tilde{w} \in |T|_{n}$.  Since $N_{w}$ is defined over $K(N_{w})$ and $\phi_{\tilde{w}}$ is defined over $K(N_{\tilde{w}}) \subseteq K(N_{w})$, it follows that $\phi_{\tilde{w}}(N_{w})$ is defined over $K(N_{w})$.  Now the Richelot isogeny $\psi_{w} : J_{\tilde{w}} \to J_{w}$ is defined over $K_{\tilde{w}}(\phi_{\tilde{w}}(N_{w})) \subseteq K(N_{w})$ by Theorem \ref{prop: Richelot isogeny}(a), so $\phi_{w} = \psi_{w} \circ \phi_{\tilde{w}}$ is defined over $K(N_{w})$.  Moreover, since $J_{w}$ is the image of $J_{\tilde{w}}$ under $\psi_{w}$, $J_{w}$ is defined over $K(N_{w})$.  Since $K_{w}$ is the extension of $K$ over which $C_{w}$ (and hence also $J_{w}$) are defined, it follows that $K_{w} \subseteq K(N_{w})$, thus proving part (a).

Now part (a) and the fact that $K_{n}'$ is the compositum of the fields $K_{w}$ for all $w \in |T|_{\leq n} \backslash \{w_{0}\}$ imply that $K_{n}'$ is contained in the compositum of the extensions $K(N_{w})$ for all $w \in |T|_{\leq n} \backslash \{w_{0}\}$.  Since $\{N_{w}\}_{w \in |T|_{\leq n}}$ clearly generates $J[2^{n}]$, this compositum is $K_{n}$.  Thus, $K_{n}' \subseteq K_{n}$, which is the statement of (b).

\end{proof}

For any Galois element $\sigma \in \Gal(\bar{K} / K)$ and any element $R = [(R_{1}, R_{2}, R_{3})] \in \bar{\mathcal{R}}$, we define $R^{\sigma} = [(R_{1}^{\sigma}, R_{2}^{\sigma}, R_{3}^{\sigma})] \in \bar{\mathcal{R}}$ by letting $R_{i}^{\sigma}$ be the cardinality-$2$ set obtained by letting $\sigma$ act on the elements of $R_{i}$, for $i = 1, 2, 3$ (with the convention that $\sigma$ fixes $\infty$).  It is clear that this action of $\Gal(\bar{K} / K)$ on $\bar{\mathcal{R}}$ is well defined.

Next we want to determine how the absolute Galois group of $K$ acts on the $R_{w}$'s defined above.  In order to describe this Galois action, we will adopt the following notation.  Recall that $\rho_{2} : \Gal(\bar{K} / K) \to \GSp(T_{2}(J))$ is the continuous homomorphism induced by the natural Galois action on $T_{2}(J)$.  Moreover, for each $n \geq 0$, let $\bar{\rho}_{2}^{(n)} : \Gal(K_{n} / K) \to \GSp(J[2^{n}])$ denote the homomorphism induced by the natural Galois action on $J[2^{n}]$.  Meanwhile, the automorphism group $\Aut_{\qq_{2}}(V_{2}(J))$ acts on the set of rank-$4$ $\zz_{2}$-lattices in $V_{2}(E)$ by left multiplication.  The Galois equivariance of the Weil pairing implies that the image of $\rho_{2}$ in $\Aut_{\qq_{2}}(V_{2}(J))$ acts on $|S|$.  It is straightforward to check that this action preserves adjacency of the vertices.  Thus, the Galois group acts on the set of non-backtracking paths starting at $v_{0}$ in $S$, and so it acts on the universal covering graph $T$.  We denote this action $\Gal(\bar{K} / K) \times |T| \to |T|$ by $(\sigma, w) \mapsto w^{\sigma}$.  Similarly, the image of $\rho_{2}^{(n)}$ in $\GSp(J[2^{n}])$ acts on $|T|_{\leq n}$, since non-backtracking paths of length $n$ in $S$ correspond to series of subgroups of $J[2^{n}]$.  Thus, the group $\Gal(K_{n} / K)$ acts on $|T|_{\leq n}$, and we denote this action $\Gal(K_{n} / K) \times |T|_{\leq n} \to |T|_{\leq n}$ by $(\sigma, w) \mapsto w^{\sigma | K_{n}}$.  Note that it follows from the construction of the map $w \mapsto N_{w}$ that $N_{w}^{\sigma} = N_{w^{\sigma}}$, and that if $w \in |T|_{\leq n}$, then $N_{w} < J[2^{n}]$ and $N_{w}^{\sigma |_{K_{n}}} = N_{w^{\sigma | K_{n}}}$.

\begin{lemma}\label{prop: Galois action on decoration genus 2}

For any $\sigma \in \Gal(\bar{K} / K)$ and $w \in |T|$, we have $R_{w}^{\sigma} = R_{w^{\sigma}}$ up to permutation equivalence.  If $w \in |T|_{\leq n}$, then $R_{w}^{\sigma |_{K_{n}}} = R_{w^{\sigma | K_{n}}}$.

\end{lemma}

\begin{proof}

Choose any $\sigma \in \Gal(\bar{K} / K)$.  We will prove that $R_{w}^{\sigma} = R_{w^{\sigma}}$ for all $w \in |T|_{n}$ for each $n \geq 1$.

First, let $w \in |T|_{1}$.  Then $N_{w}$ is the maximal Weil isotropic subgroup of $J[2]$ corresponding to $R_{w}$; in other words, if $R_{w} = (R_{1}, R_{2}, R_{3})$, then $N_{w} = \{e_{\varnothing}, e_{R_{1}}, e_{R_{2}}, e_{R_{3}}\}$, using the notation of \S3.  Note that $e_{R_{i}}^{\sigma} = e_{R_{i}^{\sigma}}$.  So 
\begin{equation}\label{proving Galois action on decoration genus 2}
N_{w^{\sigma}} = N_{w}^{\sigma} = \{e_{\varnothing}, e_{R_{1}^{\sigma}}, e_{R_{2}^{\sigma}}, e_{R_{3}^{\sigma}}\},
\end{equation}
and $R_{w}^{\sigma} = [(R_{1}^{\sigma}, R_{2}^{\sigma}, R_{3}^{\sigma})]$ is the corresponding element of $\bar{\mathcal{R}}$.  This proves the first statement for $n = 1$.  Moreover, from the construction in Theorem \ref{prop: Richelot isogeny}, this implies that $C_{w^{\sigma}} = C_{w}^{\sigma}$, and so $J_{w^{\sigma}} = J_{w}^{\sigma}$.  Moreover, one can check from the explicit definition of the Richelot isogeny in \S3 that $\phi_{w^{\sigma}} = \phi_{w}^{\sigma}$.

Now choose $n \geq 2$ and assume inductively that for all $w \in |T|_{n - 1}$, $R_{w^{\sigma}} = R_{w}^{\sigma}$, as well as the analogous statements for $C_{w^{\sigma}}$, $J_{w^{\sigma}}$, and $\phi_{w^{\sigma}}$.  Choose $w \in |T|_{n}$; then $\tilde{w} \in |T|_{n - 1}$ and we may apply the inductive assumptions to $\tilde{w}$.  Then we have 
\begin{equation}\label{proving Galois action on decoration2 genus 2}
\phi_{\tilde{w}^{\sigma}}(N_{w^{\sigma}}) = \phi_{\tilde{w}}^{\sigma}(N_{w}^{\sigma}) = (\phi_{\tilde{w}}(N_{w}))^{\sigma}.
\end{equation}
So $\psi_{w^{\sigma}}$ is the Richelot isogeny corresponding to the maximal Weil isotropic subgroup $(\phi_{\tilde{w}}(N_{w}))^{\sigma} < J_{\tilde{w}}^{\sigma}[2]$.  Then by a similar argument as was used in the $n = 1$ case, $R_{w^{\sigma}} = R_{w}^{\sigma}$.  Again, from the construction in Theorem \ref{prop: Richelot isogeny}, this implies that $C_{w^{\sigma}} = C_{w}^{\sigma}$, and so $J_{w^{\sigma}} = J_{w}^{\sigma}$.  Moreover, again one can check from the explicit definition of the Richelot isogeny in \S3 that $\psi_{w^{\sigma}} = \psi_{w}^{\sigma} : J_{\tilde{w}}^{\sigma} \to J_{w}^{\sigma}$.  Thus, $\phi_{w^{\sigma}} = \psi_{w^{\sigma}} \circ \phi_{\tilde{w}^{\sigma}} = \psi_{w}^{\sigma} \circ \phi_{\tilde{w}}^{\sigma} = \phi_{w}^{\sigma}$, as desired.

Now let $w \in |T|_{\geq n}$.  Then $R_{w}$ is fixed by all elements of $\Gal(\bar{K} / K_{n})$, and one checks from the definitions that $w_{\sigma} = w_{\sigma | K_{n}}$ for any $\sigma \in \Gal(\bar{K} / K)$.  Thus, the second statement follows from the first.

\end{proof}

It is an immediate corollary of Proposition \ref{prop: field of definition of isogenies genus 2}(b) that $\Gal(K_{n}' / K_{1}')$ is a quotient of $\Gal(K_{n} / K_{1})$ for all $n \geq 1$.  The following key proposition characterizes the kernel $\Gal(K_{n} / K_{n}')$.

\begin{prop}\label{prop: key lemma}

For all $n \geq 1$, the image of $\Gal(K_{n} / K_{n}')$ under $\bar{\rho}_{2}^{(n)}$ coincides with the subgroup of scalar automorphisms in $\bar{G}^{(n)}$.

\end{prop}

\begin{proof}

Fix $n \geq 1$.  Since $K_{n}' \supseteq K_{1}$ for each $n \geq 1$, we only need to consider the Galois subgroup $\Gal(K_{n} / K_{1})$.  Part (a) of Proposition \ref{prop: characterization of fields' genus 2}, with the help of Lemma \ref{prop: Psi is a decoration genus 2}, implies that $K_{n}'$ is generated over $K_{1}$ by the entries of the matrices $M(R_{w})$ for all $w \in |T|_{\leq n} \setminus \{w_{0}\}$.  Note that a Galois automorphism fixes all the entries of $M(R_{w})$ if and only if it fixes $R_{w}$.  Therefore, the elements of $\Gal(K_{\infty} / K_{1})$ which fix $K_{n}'$ are exactly those which fix all of the permutation equivalence classes $R_{w} \in \bar{\mathcal{R}}$ for all $w \in |T|_{\leq n} \backslash \{v_{0}\}$.  Lemma \ref{prop: Galois action on decoration genus 2} implies that, for any $\sigma \in \Gal(K_{n} / K_{1})$, $R_{w}^{\sigma |_{K_{n}}} = R_{w^{\sigma | K_{n}}}$, so the Galois automorphisms in $\Gal(K_{n} / K_{1})$ which fix $R_{w}$ for all $w \in |T|_{\leq n}$ are the ones sent by $\bar{\rho}_{2}^{(n)}$ to the elements of $\GSp(J[2^{n}])$ that fix all vertices in $|T|_{\leq n}$.  Let $\zeta \in \GSp(J[2^{n}])$ be such an automorphism.  Then all maximal Weil isotropic subgroups of $J[2^{n}]$ are stable under $\zeta$.  Let $P$ be an element of order $2^{n}$ in $J[2^{n}]$.  Then $\zeta$ stabilizes the intersection of all maximal Weil isotropic subgroups of $J[2^{n}]$ which contain $P$, which is $\langle P \rangle$.  So $\zeta$ takes $P$ to an odd scalar multiple of $P$ for all $P$ of order $2^{n}$ in $J[2^{n}]$ (and hence for all $P \in J[2^{n}]$).  But the endomorphisms of the $\zz / 2^{n}\zz$-module $J[2^{n}]$ which take every element to an odd scalar multiple of itself are scalar automorphisms in $\GSp(J[2^{n}])$.  Conversely, scalar automorphisms of $G$ fix all vertices of $T$, and so $\bar{\rho}_{2}^{(n)}$ maps $\Gal(K_{n} / K_{n}')$ onto the subgroup of scalar automorphisms in $G$.

\end{proof}

From now on, for ease of notation, we set $a_{i,j} := \alpha_{i} - \alpha_{j}$ for $1 \leq i < j \leq 5$.  For each $a_{i,j}$, choose an element $\sqrt{a_{i,j}} \in \bar{K}$ whose square is $a_{i,j}$.  Also, for $r \in \zz_{2}$, we will write $r \in \GSp(T_{2}(J))$ (resp. $r \in \GSp(J[2^{n}])$ for some $n$) for the scalar endomorphism of $\GSp(T_{2}(J))$ (resp. $\GSp(J[2^{n}])$) which acts by multiplying each element by $r$ (resp. $r$ modulo $2^{n}$).  Recall that $G$ contains the principal congruence subgroup $\Gamma(2) \lhd \Sp(T_{2}(J))$, and therefore, $\{\pm 1\} \lhd G \cap \Sp(T_{2}(J))$.

The following proposition, together with Proposition \ref{prop: characterization of fields' genus 2}(b), gives essentially the statement of Theorem \ref{prop: main theorem genus 2}.

\begin{prop}\label{prop: main quadratic extension genus 2}

The subextension $K_{\infty}'(\mu_{2}) \supset K(\mu_{2})$ corresponds to the subgroup $\{\pm 1\} \lhd G \cap \Sp(T_{2}(J)) \cong \Gal(K_{\infty} / K(\mu_{2}))$.  In fact, 
\begin{equation}\label{main quadratic extension} K_{\infty} = K_{\infty}'(\sqrt{a_{i, j}})(\mu_{2}), \end{equation}
for any $1 \leq i < j \leq 5$, and the Galois automorphism corresponding to $-1$ acts by sending $\sqrt{a_{i,j}}$ to $-\sqrt{a_{i,j}}$.

\end{prop}

\begin{proof}

If we replace $K$ with $K(\mu_{2})$, it will suffice to assume that $K$ contains all $2$-power roots of unity and to prove that $K_{\infty} = K_{\infty}'(\sqrt{a_{i, j}})$ for any $1 \leq i < j \leq 5$, and that the Galois element corresponding to $-1 \in \Sp(T_{2}(J))$ acts as claimed.

Let $\sigma$ be a Galois automorphism of $K_{\infty}$ over $K_{\infty}'$.  By Proposition \ref{prop: key lemma}, $\bar{\rho}_{2}^{(n)}$ must be a scalar automorphism of $\mathrm{Sp}(J[2^{n}])$ for every $n$ and therefore must be a scalar automorphism in $\Sp(T_{2}(J))$.  But the only nonidentity scalar matrix in $\Sp(T_{2}(J))$ is $-1$.  Conversely, Proposition \ref{prop: key lemma} implies that $-1 \in \Gal(K_{\infty} / K_{\infty}')$, hence the first statement.

It immediately follows that $K_{\infty}$ is generated over $K_{\infty}'$ by any element of $K_{\infty}$ which is not fixed by the Galois automorphism $\sigma$ such that $\rho_{2}(\sigma) = -1 \in \Sp(T_{2}(J))$.  Clearly, $\bar{\rho}_{2}^{(2)}(\sigma |_{K_{2}}) = -1 \in \Gamma(2) / \Gamma(4)$.  But setting $n = 2$ in the statement of Proposition \ref{prop: key lemma} implies that 
\begin{equation}\label{Gal(K_{2} / K_{2}') genus 2}
\Gal(K_{2} / K_{2}') \cong \{\pm 1\} \lhd \Gamma(2) / \Gamma(4),
\end{equation}
so any element in $K_{2} \backslash K_{2}'$ will not be fixed by $-1$.  In order to find such an element, we proceed to compute generators for $K_{2}'$ over $K$.  Choose $w \in |T|_{1}$ such that $R_{w}$ is permutation equivalent to $\{\{\alpha_{1}, \alpha_{2}\}, \{\alpha_{3}, \alpha_{4}\}, \{\alpha_{5}, \infty\}\}$.  A tedious but straightforward computation shows that $|\Ri(R_{w})|$, which is the set of Weierstrass roots of the degree-$6$ curve $C_{w}$, consists of the elements 
$$\frac{-\alpha_{1}\alpha_{2} + \alpha_{3}\alpha_{4} \pm \sqrt{(\alpha_{1} - \alpha_{3})(\alpha_{1} - \alpha_{4})(\alpha_{2} - \alpha_{3})(\alpha_{2} - \alpha_{4})}}{-\alpha_{1} - \alpha_{2} + \alpha_{3} + \alpha_{4}},$$
$$\alpha_{5} \pm \sqrt{(\alpha_{1} - \alpha_{5})(\alpha_{2} - \alpha_{5})},\ \ \alpha_{5} \pm \sqrt{(\alpha_{3} - \alpha_{5})(\alpha_{4} - \alpha_{5})}.$$
It follows that $K_{1}(\{R_{u})\}_{u \in |T|_{2}, w = \tilde{u}})$ is the extension of $K_{1}$ obtained by adjoining the square root terms in the above expressions.  By permuting the $\alpha_{i}$'s and using the fact that $K_{1}(\{R_{u}\}_{u \in |T|_{2}}) = K_{2}'$ by Proposition \ref{prop: characterization of fields' genus 2}(a), we see that 
\begin{equation}\label{description of K_{2}' genus 2}
K_{2}' = K_{1}(\{\sqrt{(a_{i,j}a_{l,m})}\}_{\{i, j\} \neq \{l, m\}}).
\end{equation}
Choose $i, j$ such that $1 \leq i < j \leq 5$.  It follows from (\ref{description of K_{2}' genus 2}) that $\sqrt{a_{i,j}} \notin K_{2}'$, although $\sqrt{a_{i,j}} \in K_{2}$ by Proposition 2.4 of \cite{yelton2014images}.  So $\sqrt{a_{i,j}} \in K_{\infty} \setminus K_{\infty}'$ and can be used to generate $K_{\infty}$ over $K_{\infty}'$.  Since $-1$ does not fix $\sqrt{a_{i,j}}$ but does fix its square $a_{i,j} \in K_{1} \subset K_{\infty}'$, it follows that $-1$ acts by sending $\sqrt{a_{i,j}}$ to $-\sqrt{a_{i,j}}$.

\end{proof}

\section*{Acknowledgments}

The author would like to thank Yuri Zarhin for discussions that were extremely helpful in producing this work.  The author would also like to thank Mihran Papikian for suggesting revisions that helped to improve the exposition.

\bibliographystyle{plain}
\bibliography{bibfile}

\end{document}